\documentclass[10pt,dvipdfmx]{extarticle}
\usepackage{amsmath,amsthm,amssymb,amsbsy,amsopn,mathtools,comment,bbm,enumerate,mathrsfs}
\usepackage{graphicx,color}
\usepackage[bookmarks=true,bookmarksnumbered=true,bookmarkstype=toc]{hyperref}
\usepackage{mathptmx}
\usepackage{algorithm,bm,makecell,pict2e,diagbox}
\usepackage{algorithmicx}
\usepackage{epstopdf,colortbl,multirow}
\usepackage{tikz,subfig,pgfplots}
\usetikzlibrary{matrix,arrows,intersections,shapes,plotmarks,positioning}
\usetikzlibrary{positioning}
\pgfplotsset{compat=newest}

\newtheorem{Def}{Definition}[section]
\newtheorem{theorem}[Def]{Theorem}

\newtheorem{proposition}[Def]{Proposition}

\newtheorem{assumption}[Def]{Assumption}

\makeatletter
\@addtoreset{equation}{section}
\makeatother

\title{A general approach to sample path generation of infinitely divisible processes via shot noise representation
\footnotetext[0]{
Email: raykawai@g.ecc.u-tokyo.ac.jp. Postal address: Graduate School of Arts and Sciences / Mathematics and Informatics Center, The University of Tokyo, Japan. The author would like to thank the reviewers and Sida Yuan for valuable suggestions and comments.}}
\author{\sc Reiichiro Kawai}
\date{}

\begin{document}

\maketitle

\begin{abstract}
\noindent 
We establish a sample path generation scheme in a unified manner for general multivariate infinitely divisible processes based on shot noise representation of their integrators.
The approximation is derived from the decomposition of the infinitely divisible process to three independent components based on jump sizes and timings: the large jumps over a compact time interval, small jumps over the entire time interval and large jumps over an unbounded time interval.
The first component is taken as the approximation and is much simpler than simulation of general Gaussian processes, while the latter two components are analyzed as the error.
We derive technical conditions for the two error terms to vanish in the limit and for the scaled component on small jumps to converge to a Gaussian process so as to enhance the accuracy of the weak approximation.
We provide an extensive collection of examples to highlight the wide practicality of the proposed approach.
	
\vspace{0.3em}
	
\end{abstract}

\section{Introduction}

The growing appearance of infinitely divisible processes in applied contexts naturally fuels an increasing demand for simulation methods.
Classically, the go-to paradigm for sample path generation of a stochastic process is to partition the time interval into deterministic sample points and recursively simulate the stochastic process from one sample point to the next via increments.
This may be appropriate in the case where the driving process is Gaussian, such as the fractional Brownian motion 
and the Gaussian Ornstein--Uhlenbeck process. 
However, in the case where the integrator has jump components, simulation by increments may not be viable due to a lack of efficient sampling methods for the corresponding infinitely divisible increments. 
Applied contexts which demand jump dynamics often also prefer the observation of individual jumps in sample paths. 
For example, in insurance mathematics, 
observation of individual jumps is desired for the computation of ruin times.
As sample path generation by increments cannot portray individual jumps, the paradigm falls out of favor. 

In the recent decades, an increasingly popular simulation method for L\'evy processes without Gaussian components is through the truncation of their shot noise series representation \cite{imai2011finite, imai2013numerical}.
As such representations of the L\'evy process are decompositions in terms of its individual jumps \cite{rosinski2001series}, this truncation approach fulfills the desire for a jump-based method. The extension of this truncation method for infinitely divisible processes driven by integrators without Gaussian components has made appearances in the literature based on shot noise representations via a decomposition by jump timings and sizes and the corresponding error analysis \cite{kawai2016higher, kawai2017sample}.
In short, the idea is to decompose the infinitely divisible process into three independent components: (i) the simulatable component corresponding to large jumps of the integrator within a bounded time interval, (ii) the component corresponding to small jumps which may be approximated by a Gaussian process, and (iii) the remaining component corresponding to large jumps outside of the bounded time interval. Error analysis is performed by investigating the latter two components.

In this paper, we develop the general framework of simulating infinitely divisible processes driven by L\'evy processes without Gaussian components based on their shot noise representations.
The proposed framework unifies and generalizes the existing case-by-case approaches, as summarized in \cite{2021arXiv210110533Y}, and offers a sample path generation scheme that is indeed a lot easier than simulation of general Gaussian processes.
We provide technical conditions for error analysis in terms of the integrator and the deterministic kernel and include an extensive collection of infinitely divisible processes of interest in the literature so as to exemplify the effectiveness of the proposed approach.
To maintain the flow of the paper, we collect all proofs in the Appendix.


\section{Preliminaries}\label{section preliminaries}

In what follows, we work under the probability space $(\Omega,\mathscr{F},\mathbb P)$.
Let $\mathbb N\coloneqq\{1,2,\cdots\}$ and $\mathbb N_0 \coloneqq\{0,1,2,\cdots\}$. 
Denote 
$\mathbb R^d_0\coloneqq\mathbb R^d\backslash\{0\}$ and by ${\rm Leb}(D)$ the Lebesgue measure of a set $D$.
The notation $\{\Gamma_k\}_{k\in\mathbb N}$ will be reserved throughout to denote the arrival times of the standard Poisson process.

We next review some relevant essential notions of L\'evy processes, their shot noise representations and infinitely divisible processes.
Recall that by the L\'evy-Khintchine formula, an infinitely divisible random vector without a Gaussian component has a characteristic function of the form
\begin{equation}\label{cf}
\varphi({\bf y}) = \exp\left[i\langle {\bf y},{\bf c}\rangle + \int_{\mathbb R^d}\left(e^{i\langle {\bf y},\mathbf z\rangle} - 1 - i\langle {\bf y},\mathbf z\rangle \mathbbm1_{(0,1]}(\|\mathbf z\|)\right)\, \nu(d\mathbf z)\right],\quad {\bf y}\in\mathbb R^d,
\end{equation}
where ${\bf c}\in\mathbb R^d$ corresponds to a drift component and the measure $\nu$ is referred to as a \textit{L\'evy measure} which satisfies the integrability condition $\int_{\mathbb R^d_0}(1\wedge\|\mathbf z\|^2)\, \nu(d\mathbf z)<+\infty$.
A stochastic process is infinitely divisible if its finite dimensional distributions are infinitely divisible. 
All L\'evy processes are thus infinitely divisible, whereas the class of infinitely divisible processes is more general. 
With a set of time indices $\mathcal T\subseteq\mathbb R$, let $\{X_t:t\in[0,T]\}$ be a stochastic process in $\mathbb R^d$, 
described by the stochastic integral
\begin{equation}\label{Levy-driven stochastic integral}
	X_t = \int_{\mathcal T}f(t,s)\,dL_s,\quad t\in [0,T],
\end{equation}
where $\{L_s:\,s\in\mathcal T\}$ is a L\'evy process in $\mathbb R^{d}$ over time $\mathcal T$, whose unit-time marginal is characterized by \eqref{cf}. The focus of this paper is on the simulation of such infinitely divisible processes.
Each jump of the underlying L\'evy process at time $s$ is modulated by a mapping $f$ from $[0,T]\times \mathcal{T}$ to a suitable space, so we have a L\'evy-It\^o decomposition of the form
\begin{equation}\label{levy ito for levy}
X_t = \int_{\mathcal T}\int_{\mathbb R_0^{d}} f(t,s)\mathbf z(\mu - \nu\mathbbm1_{(0,1]}(\|\mathbf z\|))(d\mathbf z,ds), \quad t\in[0,T],
\end{equation}
where $\mu(d\mathbf z,ds)$ is the Poisson random measure on $\mathbb R_0^{d}\times\mathcal T$ associated with the L\'evy process $\{L_s:s\in\mathcal T\}$ and $\nu(d\mathbf z)ds$ is the corresponding compensator.

From the L\'evy-It\^o decomposition \eqref{levy ito for levy} for L\'evy processes without Gaussian components, shot noise series representations can be derived by methods of expressing the underlying Poisson random measure as a series of Dirac delta measures scattered at random points in $\mathbb R^d_0\times\mathbb [0,+\infty)$.
We provide the most general formulation \cite{rosinski2001series}, as follows.

\begin{theorem}[Generalized shot noise method]\label{theorem generalized shot noise method}
	Suppose a L\'evy measure $\nu$ on $\mathbb R^d_0$ can be decomposed as
\begin{equation}\label{generalized shot noise decomposition}
 \nu(B) = \int_{0+}^{+\infty} \mathbb P(H(r,U)\in B)\, dr, \quad B\in\mathcal B(\mathbb R_0^d),
\end{equation}
where $U$ is a random vector in some space $\mathscr U$ and $H:(0,+\infty)\times\mathscr U\to\mathbb R^d$ is such that for every $u\in\mathscr U$, $r\mapsto \|H(r,u)\|$ is nonincreasing.
Then, where $\mathcal T\subset\mathbb R$ is a bounded time interval, it holds that
	\begin{equation}\label{generalized shot noise series}
		\{L_s:s\in\mathcal T\} \stackrel{\mathscr L}{=} \left\{\sum_{k=1}^{+\infty}\left[H\left(\frac{\Gamma_k}{{\rm Leb}(\mathcal T)},U_k\right)\mathbbm1_{[\inf\mathcal T,s]}(T_k) - {\rm Leb}((\inf\mathcal T,s])c_k\right]:s\in\mathcal T\right\},
	\end{equation}
	where $\{L_s:s\in\mathcal T\}$ is the L\'evy process characterized by L\'evy-Khintchine triple $(0,0,\nu)$ over $[0,T]$, $\{\Gamma_{k}\}_{k\in\mathbb N}$ are the arrival times of the standard Poisson process, $\{U_k\}_{k\in\mathbb N}$ are iid copies of $U$, $\{T_k\}_{k\in\mathbb N}$ are iid uniform random variables over $\mathcal T$, with mutual independence of the random sequences and $\{c_k\}_{k\in\mathbb N}$ is a sequence of centers in $\mathbb R^d$ with $c_k \coloneqq \int_{k-1}^k \mathbb E[H(s,U)\mathbbm1_{(0,1]}(\|H(s,U)\|)]\,ds$ for $k\in\mathbb N$.
\end{theorem}


We remark that the decomposition \eqref{generalized shot noise decomposition} exists in a similar form  to the so-called the inverse L\'evy measure method \cite{ferguson1972representation}, whereas the resulting expression often does not offer a viable numerical method \cite{imai2013numerical}.
In fact, the decomposition \eqref{generalized shot noise decomposition} is not unique but can be employed to derive a few distinct shot noise representations of an infinitely divisible random vector via, not only the aforementioned inverse L\'evy measure method but also, the rejection, thinning and Bondesson's methods \cite{2021arXiv210110533Y}.
Thanks to the non-uniqueness, each of the three typical L\'evy measures (Section \ref{section Levy measures}) admits at least one implementable expression for simulation purposes \cite{imai2011finite, kawai2017sample}.


\section{Decomposition based on jump sizes and timings}\label{section decomposition}

Consider the infinitely divisible process $\{X_t:t\in[0,T]\}$ in $\mathbb R^d$, described by the stochastic integral \eqref{Levy-driven stochastic integral}.
We look to truncate the L\'evy measure based on a shot noise representation of the integrator to obtain simulatable and error components. 
To this end, we exclusively consider the case when the L\'evy measure $\nu$ is infinite, as otherwise, exact methods are readily available.

\begin{assumption}\label{standing assumptions}
$\nu(\mathbb{R}_0^d)=+\infty$.
\end{assumption}

To illustrate the significance of this assumption, consider a subordinator $\{L_s:\,s\in [0,1]\}$ with $c_k\equiv 0$ and $\nu(dz)=e^{-z}dz$ on $(0,+\infty)$, violating Assumption \ref{standing assumptions} with $\nu(0,+\infty)=1$.
We then have $H(r,u)=-\ln (r)$, which has only to be defined on $(0,1]$, irrespective of the second argument.
Hence, with $\mathcal{T}=[0,1]$, a shot noise representation in the framework \eqref{generalized shot noise series} is given by 
\begin{equation}\label{compound poisson}
 \{L_s:s\in [0,1]\} \stackrel{\mathscr L}{=} \left\{-\sum_{\{k\in \mathbb{N}:\,\Gamma_k\le 1\}}\ln (\Gamma_k)\mathbbm{1}_{[0,s]}(T_k):\,s\in [0,1]\right\},
\end{equation}
which is indeed an exact simulation method without the need for truncation.
Assumption \ref{standing assumptions} implies that
\begin{equation}\label{H tends to 0}
 \lim_{r\to +\infty}\|H(r,U(\omega))\|=0,
\end{equation}
for almost every $\omega \in \Omega$, since the origin is the only possible source of infinite mass of the L\'evy measure.
Note that the converse is not true.
For instance, with the aforementioned example of a subordinator with $\nu(dz)=e^{-z}dz$ on $(0,+\infty)$ and $H(r,u)=-\ln (r)$ on $(0,1]$, it still holds that $|H(r,u)|\downarrow 0$ as $r\uparrow 1$.

As in the case of L\'evy processes, shot noise representations for infinitely divisible processes converge almost surely uniformly \cite{basse2013uniform}. Moreover, if the probability space is rich enough, then one can choose the random sequences such that the shot noise representation is almost surely equal to the infinitely divisible process \cite{rosinski2018representations}.

Now, let $\{\mathcal T_n\}_{n\in\mathbb N}$ denote a nondecreasing sequence of connected open intervals over the time indices such that $\cup_{n\in\mathbb N}\mathcal T_n = \mathcal T$ and $\text{Leb}(\mathcal T_n) < +\infty$, so $n$ represents a truncation based on jump timings. Suppose a shot noise representation of the driving L\'evy process $\{L_s : s\in\mathcal T_n\}$ is given by \eqref{generalized shot noise series} with the time set $\mathcal T_n$ in lieu of $\mathcal T$.
For each $m\in \mathbb{N}$, denote by $\nu_m$ the finite L\'evy measure defined by
\begin{equation}\label{def nu_m}
 \nu_m(B):= \int_{0+}^m \mathbb P(H(r,U)\in B)\, dr, \quad B\in\mathcal B(\mathbb R_0^d),
\end{equation}
as a truncation of the decomposition \eqref{generalized shot noise decomposition}, corresponding to the Poisson truncation approximation
\begin{equation}\label{truncated series}
 \left\{\sum_{\{k\in\mathbb N:\Gamma_k\le \text{Leb}(\mathcal T_n)\nu_m(\mathbb{R}_0^d)\}}\left[H\left(\frac{\Gamma_k}{\text{Leb}(\mathcal T_n)},U_k\right)\mathbbm1_{[\inf\mathcal T_n,s]}(T_k) - \text{Leb}([\inf\mathcal T_n,s])c_k\right]:s\in\mathcal T_n\right\},
\end{equation}
via a similar truncation to the representation \eqref{compound poisson} on the basis of the Poison arrival times.
In view of \eqref{def nu_m}, the sequence of finite L\'evy measures $\{\nu_m\}_{m\in\mathbb N}$ increases to the original infinite L\'evy measure $\nu$ as $m\to+\infty$, in the sense of $\nu_m(\mathbb{R}_0^d)=\int_{0+}^m\mathbb P(H(r,U)\in \mathbb{R}_0^d)dr=m$ for all $m\in \mathbb{N}$.
Hence, we hereafter write $m$ instead of $\nu_m(\mathbb{R}_0^d)$ in the region of the summation \eqref{truncated series}.
It is worth stressing again \cite[Theorem 3.1]{imai2013numerical} that all shot noise representations, including the inverse L\'evy measure, rejection, thinning and Bondesson's methods, are well in the framework of the finite truncation \eqref{def nu_m}.

Armed with the truncation parameters $m$ and $n$ in \eqref{truncated series}, we decompose the infinitely divisible process as
\begin{equation}\label{stochastic integral decomposition}
	X_t = X_t(m,n) + Q_t(m) + R_t(m,n), \quad t\in[0,T],
\end{equation}
where
\begin{align}
	\notag X_t(m,n) &\coloneqq \int_{\mathcal T_n}\int_{\mathbb R_0^{d}}f(t,s)\mathbf z\,(\mu_m - \nu_m\mathbbm1_{(0,1]}(\|\mathbf z\|))(d\mathbf z,ds),\\
	\notag Q_t(m) &\coloneqq \int_{\mathcal T}\int_{\mathbb R_0^{d}}f(t,s)\mathbf z\,((\mu - \mu_m) - (\nu - \nu_m)\mathbbm1_{(0,1]}(\|\mathbf z\|))(d\mathbf z,ds),\\
	\notag R_t(m,n) &\coloneqq \int_{\mathcal T\backslash\mathcal T_n}\int_{\mathbb R_0^{d}}f(t,s)\mathbf z\,(\mu_m - \nu_m\mathbbm1_{(0,1]}(\|\mathbf z\|))(d\mathbf z,ds),
\end{align}
where $\mu_m(d\mathbf z,ds)$ denotes a Poisson random measure with intensity measure $\nu_m(d\mathbf z)ds$.
The stochastic processes on the right hand side of \eqref{stochastic integral decomposition} are mutually independent, thanks to the independent scattering of the Poisson random measure $\mu(d{\bf z},ds)$. 
As $\text{Leb}(\mathcal T_n)\nu_m(\mathbb R_0^{d})<+\infty$ for every $(m,n)\in\mathbb N^2$, the first component $\{X_t(m,n):t\in[0,T]\}$ can be treated as the principal component of the approximation by the Poisson truncation of shot noise representation.
The remaining two components $\{Q_t(m):t\in[0,T]\}$ and $\{R_t(m,n):t\in[0,T]\}$ are treated as the error processes for further analysis.
The stochastic process $\{Q_t(m):t\in[0,T]\}$ comprising of small jumps can sometimes be approximated by a Gaussian process under suitable technical conditions \cite{asmussen2001approximations, cohen2007gaussian}.
If $\text{Leb}(\mathcal T) < +\infty$, 
then one may set $\mathcal T_n \equiv \mathcal T$, that is, the component $\{R_t(m,n):t\in[0,T]\}$ vanishes. 
In fact, we often have $\mathcal T = [0,T]$, in which the temporal truncation via the sequence $\{\mathcal T_n\}_{n\in\mathbb N}$ is irrelevant, or $\mathcal T = \mathbb R$.
It is worth noting that the limiting process $\{\lim_{m\to +\infty}R_t(m,n):\,t\in [0,T]\}$, which is certainly well defined, can be considered as the residual component after time truncation alone and is discussed later in Section \ref{subsection R} when the integrator is the stable process.


\subsection{Simulatable component}

We first consider the simulatable process $\{X_t(m,n):t\in[0,T]\}$, where
\[
	X_t(m,n) = \int_{\mathcal T_n}\int_{\mathbb R_0^{d}}f(t,s)\mathbf z\,(\mu_m - \nu_m\mathbbm1_{(0,1]}(\|\mathbf z\|))(d\mathbf z,ds),\quad t\in[0,T].
\]
As the L\'evy measure $\nu_m$ is finite with $\nu_m(\mathbb{R}_0^d)=m$ and the time set $\mathcal T_n$ is bounded, the term $\{X_t(m,n):t\in[0,T]\}$ has almost surely finite number of jumps.
Since this component is thus exactly simulatable, we regard $X(m,n)$ as the principal approximation. We provide its shot noise representation as follows, which provides an exact simulation method for this component:
	\begin{equation}\label{principal truncation}
		\{X_t(m,n):t\in[0,T]\} \stackrel{\mathscr L}{=} \left\{\sum_{\{k\in\mathbb N:\Gamma_k\le \text{Leb}(\mathcal T_n)m\}}\left[f(t,T_k)H\left(\frac{\Gamma_k}{\text{Leb}(\mathcal T_n)},U_k\right) - c_k\int_{\mathcal T_n}f(t,s)\,ds\right] : t\in[0,T]\right\},
	\end{equation}
	where the sequences are the same as in \eqref{stochastic integral decomposition}.
For generating sample paths of $\{X_t(m,n):t\in[0,T]\}$ based on the Poisson truncation \eqref{principal truncation} for $J$ sample points $0=t_0<t_1<\cdots<t_{J-1}<t_J=T$ for some $J\in \mathbb{N}$, we provide the numerical recipe:
\begin{enumerate}[\bf Step 1.]
	\setlength{\parskip}{0cm}
	\setlength{\itemsep}{0cm}
	\item Generate a standard exponential random variable $E_1$. If $E_1 \le \text{Leb}(\mathcal T_n)m$, then assign $\Gamma_1 \leftarrow E_1$. Otherwise, return the degenerate zero process as the approximate sample path and terminate the algorithm.
	
	\item While $\Gamma_k \le \text{Leb}(\mathcal T_n)m$, generate a standard exponential random variable $E_{k+1}$ and assign $\Gamma_{k+1} \leftarrow \Gamma_k + E_{k+1}$. Denote this as $\{\Gamma_k\}_{k\in\{1,\cdots,N\}}$, where $N$ satisfies $\Gamma_N \le \text{Leb}(\mathcal T_n)m < \Gamma_{N+1}$.
	
	\item Generate a sequence $\{T_k\}_{k\in\{1,\cdots,N\}}$ of iid uniform random variables on $\mathcal T_n$.
	
	\item Generate a sequence $\{U_k\}_{k\in\{1,\cdots,N\}}$ of suitable iid random variables.
	
	\item For every $j\in\{1,\cdots,J\}$, assign $X_{t_j}(m,n) \leftarrow \sum_{k=1}^{N} (f_n(t_j,T_k)H(\Gamma_k/\text{Leb}(\mathcal T_n),U_k) - c_k\int_{\mathcal T_n}f(t_j,s)\,ds)$.
	
	\item Return $\{0,X_{t_1}(m,n),\cdots,X_{t_J}(m,n)\}$ as the positions of the approximate sample path at the sample times $\{0,t_1,\cdots,t_J\}$.
\end{enumerate}

The proposed scheme is built on rather elementary operations based on straightforward iterations without the need for sophisticated coding or a big matrix operation, such as the Cholesky decomposition of a thousand-dimensional covariance matrix for simulating general Gaussian processes by increments.
Note that the deterministic integral $\int_{\mathcal T_n}f(t_j,s)\,ds$ needs to be computed either exactly if possible or by approximation, which we leave to the user's discretion.

\subsection{Error components}

In what follows, we provide analysis of the error components with all proofs collected in the Appendix so as to maintain the flow.
We begin with assumptions to establish the limiting degeneracy of the error components.

\begin{assumption}{\rm \label{sufficient conditions for limiting degeneracy}
(a) ${\rm esssup}_{s\in \mathcal{T}} \|f(t,s)\|<+\infty$ for $t\in [0,T]$;

\noindent (b) $\int_{\mathcal{T}}\|f(t,s)\|^2ds<+\infty$ for $t\in [0,T]$;

\noindent (c) $\lim_{m\to +\infty}\int_m^{+\infty}\int_{\Omega}\|H(r,U(\omega))\|^2 \mathbb{P}(d\omega)dr= 0$.
\qed}\end{assumption}

\begin{proposition}\label{proposition limiting degeneracy}
	It holds under Assumption \ref{sufficient conditions for limiting degeneracy} that $\{Q_t(m):\,t\in [0,T]\}$ and $\{R_t(m,n):\,t\in [0,T]\}$ converge to degenerate zero processes on $[0,T]$ in the sense of finite dimensional distributions as $m\land n\to +\infty$.
\end{proposition}

We have provided Assumption \ref{sufficient conditions for limiting degeneracy} (a) as a sufficient condition for Proposition \ref{proposition limiting degeneracy} as well as Theorem \ref{theorem Gaussian approximation} later, whereas it is indeed an essential requirement for sample boundedness.
That is, sample paths would otherwise be almost surely unbounded on every finite interval of positive length.
In such a case, it would be nonsensical to use truncation of shot noise representation to simulate the stochastic integral process, and even misleading if done so without the knowledge of sample path unboundedness.
This is because truncation to a finite L\'evy measure will almost surely produce a bounded sample path, which would otherwise be unbounded in the absence of truncation.
Hence, it is highly advisable to check that the stochastic process is almost surely bounded over $[0,T]$ prior to generating its sample paths.

The notion of approximating small jumps by a Gaussian process \cite{asmussen2001approximations, cohen2007gaussian} to improve the quality of the approximation can carry over to the case of some infinitely divisible processes under some technical conditions. 
Define the covariance matrix
\[
\sigma_m^2 \coloneqq \int_{\mathbb R_0^{d}}\mathbf z^{\otimes2}\,(\nu-\nu_m)(d\mathbf z) = \int_m^{+\infty}\int_\Omega \left(H(r,U(\omega))\right)^{\otimes2}\,\mathbb P(d\omega)\,dr,
\]
which is finite valued under Assumption \ref{sufficient conditions for limiting degeneracy} (c), and further denote by $\sigma_m$ the lower triangular matrix of the Cholesky decomposition of the matrix $\sigma_m^2$, so that $(\sigma_m)^{\otimes2} = \sigma_m^2$.
We scale the integrator in the $Q$ component by the matrix $\sigma_m^{-1}$ to obtain
\begin{equation}\label{def of Q tilde}
\widetilde Q_t(m;\mathcal{S}) \coloneqq \int_{\mathcal S}\int_{\mathbb R_0^{d}}f(t,s)\sigma_m^{-1}\mathbf z\,((\mu - \mu_m) - (\nu - \nu_m))(d\mathbf z,ds),
\end{equation}
for some $\mathcal{S}\subseteq \mathcal{T}$, of which convergence to a Gaussian process is our present interest.
Clearly, if two time domains, say, $\mathcal{S}_1$ and $\mathcal{S}_2$ are disjoint, then the associated L\'evy measures are disjoint, that is, the two resultant stochastic processes $\{\widetilde Q_t(m_1;\mathcal{S}_1):\,t\in [0,T]\}$ and $\{\widetilde Q_t(m_2;\mathcal{S}_2):\,t\in [0,T]\}$ are independent of each other, even irrespective of $m_1$ and $m_2$.
To ease the notation, we below use the notation $\mathbb{I}_d$ for the identity matrix in $\mathbb{R}^{d\times d}$ and $[c]_B:=c\mathbbm{1}_B(\|c\|)$ for $c\in \mathbb{R}^q$ and $B\in \mathcal{B}(0,+\infty)$ and let $A\preceq B$ indicate that the matrix $A-B$ is positive semidefinite for two square matrices $A$ and $B$ of a common order.

\begin{assumption}{\rm \label{sufficient conditions for weak convergence}
(a) The matrix $\sigma^2_m$ is positive definite for sufficiently large $m$;

\noindent (b) For every $\kappa>0$, 
		$\lim_{m\to +\infty}\int_m^{+\infty}\int_{\Omega}[ \|\sigma_m^{-1}H(r,U(\omega))\|^2]_{(\kappa,+\infty)} \mathbb{P}(d\omega)dr=0$;
		
\noindent (c) There exist $c_1>0$, $c_2>0$ and $\mathcal{S}\subseteq \mathcal{T}$ such that $\int_{\mathcal{S}}(f(t_2,s)-f(t_1,s))^{\otimes 2} ds \preceq c_2(t_2-t_1)^{c_1} \mathbb{I}_d$ for $0\le t_1\le t_2\le T$.
\qed}\end{assumption}

We remark that the time domain $\mathcal{S}$ in Assumption \ref{sufficient conditions for weak convergence} (c) corresponds to that in the definition \eqref{def of Q tilde}.
Indeed, we have specified the time domain $\mathcal{S}$ in \eqref{def of Q tilde} so as to address the case where independent processes $\{\widetilde Q_t(m_1;\mathcal{S}_k):\,t\in [0,T]\}_k$ on disjoint time domains $\{\mathcal{S}_k\}_k$ satisfy Assumption \ref{sufficient conditions for weak convergence} (c) with different degrees of continuity, that is, different degrees of H\"older continuity in Theorem \ref{theorem Gaussian approximation} (iii) below.
We will later discuss such an example in Section \ref{section nfsm} of two independent components with $\mathcal{T}=\mathbb{R}$, $\mathcal{S}_1=(-\infty,0]$ and $\mathcal{S}_2=(0,+\infty)$.

We are now ready to give the results on the Gaussian approximation.
On the one hand, since the regularity of the kernel (Assumption \ref{sufficient conditions for weak convergence} (c)) does not affect the convergence of finite dimensional distributions in any way, we can retain the entire time domain $\mathcal{T}$ in Theorem \ref{theorem Gaussian approximation} (i).
In Theorem \ref{theorem Gaussian approximation} (ii) and (iii), on the other hand, we address the weak convergence of sample paths, thus the time domain needs to be restricted to where Assumption \ref{sufficient conditions for weak convergence} (c) holds.
It is also worth mentioning that Assumption \ref{sufficient conditions for weak convergence} (a) and (b) are not only sufficient conditions but also necessary conditions for Theorem \ref{theorem Gaussian approximation} to hold true \cite{cohen2007gaussian}.
Note that $\mathcal{D}([0,T];\mathbb{R}^d)$ denotes the space of c\`adl\`ag functions from $[0,T]$ to $\mathbb{R}^d$ endowed with the Skorohod topology.

\begin{theorem}\label{theorem Gaussian approximation}
Let Assumption \ref{sufficient conditions for limiting degeneracy} hold.

\noindent (i) It holds under Assumption \ref{sufficient conditions for weak convergence} (a)-(b) that $\{\widetilde{Q}_t(m;\mathcal{T}):\,t\in [0,T]\}$ converges to $\{\int_{\mathcal{T}}f(t,s)dB_s:\,t\in [0,T]\}$ in the sense of finite dimensional distributions, as $m\to +\infty$.
	
\noindent (ii) It holds that under Assumption \ref{sufficient conditions for weak convergence} (a)-(c) that $\{\widetilde{Q}_t(m;\mathcal{S}):\,t\in [0,T]\}$ converges to $\{\int_{\mathcal{S}}f(t,s)dB_s:\,t\in [0,T]\}$ in $\mathcal{D}([0,T];\mathbb{R}^d)$, as $m\to +\infty$.
	
\noindent (iii) Assume there exists a continuous version of $\{\widetilde{Q}_t(m;\mathcal{S}):\,t\in [0,T]\}$ for $m\in\mathbb{N}$ and let Assumption \ref{sufficient conditions for weak convergence} (a)-(c) hold with the exponent $c_1$ in (c) strictly greater than $d$.
Then, the weak convergence in (ii) can be replaced with the weak convergence in $\mathcal{C}([0,T];\mathbb{R}^d)$ and the limiting process is almost surely locally H\"older continuous with exponent in $(0,(c_1-d)/2)$.
\end{theorem}

Note that Assumption \ref{sufficient conditions for weak convergence} (c) is the only regularity condition on the kernel throughout.
Despite sample paths of an infinitely divisible process without Gaussian components cannot be smoother than its kernel \cite{rosinski1989path}, the smoothness of the kernel is rather naturally irrelevant for the Gaussian approximation 
in the sense of finite dimensional distributions (Theorem \ref{theorem Gaussian approximation} (i)).

Looking closely at Assumptions \ref{sufficient conditions for limiting degeneracy} and \ref{sufficient conditions for weak convergence}, the analysis of the error terms depends on the integrator and the kernel, while interestingly it suffices to investigate those two factors separately, which we do in Sections \ref{section Levy measures} and \ref{section examples of ID processes}, respectively.

\section{L\'evy measures}\label{section Levy measures}

In this section, we illustrate three typical integrators against Assumption \ref{sufficient conditions for limiting degeneracy} (c) and Assumption \ref{sufficient conditions for weak convergence} (a) and (b).
To avoid overloading the paper, we omit nonessential details in some instances by referring to relevant work in the literature.

\subsection{Gamma law} \label{example gamma law}
We start with the one-dimensional L\'evy measure $\nu(dz)=ae^{-\beta z}/zdz$ on $(0,+\infty)$
with $a>0$ and $\beta>0$, corresponding to the gamma law.
The preferred shot noise representation for numerical purposes is the one by Bondesson's method, based on the decomposition \eqref{generalized shot noise decomposition}
with $H(r,u)=\beta^{-1}e^{-r/a}u$ where $U$ is the standard exponential random variable.
Then, we have
\[
 \sigma_m^2=\int_m^{+\infty}\int_{\Omega}\left(H(r,U(\omega))\right)^2\mathbb{P}(d\omega)dr=\int_m^{+\infty}\int_{\Omega}\left(\frac{1}{\beta}e^{-r/a}U(\omega)\right)^2\mathbb{P}(d\omega)dr=\frac{a}{\beta^2}e^{-2m/a},\quad m\in\mathbb{N},
\]
which justifies Assumption \ref{sufficient conditions for limiting degeneracy} (c) and Assumption \ref{sufficient conditions for weak convergence} (a).
Assumption \ref{sufficient conditions for weak convergence} (b) is however violated, since for every $\kappa>0$, 
\[
 \int_m^{+\infty}\int_{\Omega}\left[\left|\sigma_m^{-1}H(r,U(\omega))\right|^2\right]_{(\kappa,+\infty)}\mathbb{P}(d\omega)dr
 =\int_{0+}^{+\infty}\int_{\Omega}\left[\frac{1}{a}e^{-2r/a}(U(\omega))^2\right]_{(\kappa,+\infty)}\mathbb{P}(d\omega)dr>0,
\]
which no longer depends on the index $m$.
The Gaussian approximation 
(Theorem \ref{theorem Gaussian approximation}) fails, which supports, as mentioned earlier, that Assumption \ref{sufficient conditions for weak convergence} (b) is not only sufficient but also necessary.
We refer the reader to \cite[Section 5.4]{kawai2017sample} for many other shot noise representations of the gamma law 
as well as error analysis.

\subsection{Stable law} \label{example stable law}
Consider the L\'evy measure of a stable law
\begin{equation}\label{stable Levy measure}
	\nu(B) = \int_{S^{d-1}}\int_{0+}^{+\infty}\mathbbm1_B(r\bm\xi)\frac{\alpha}{r^{\alpha+1}}q(r,{\bm \xi})dr\, \lambda (d\bm\xi),\quad B\in\mathscr B(\mathbb R^d_0),
\end{equation}
with $\alpha\in(0,2)$, a finite measure $\lambda$ on the unit sphere $S^{d-1}$ and $q\equiv 1$.
We have $H(r,{\bf u})=(r/\|\lambda\|)^{-1/\alpha}{\bf u}$, where $\|\lambda\|:=\lambda(S^{d-1})$ and $U$ is a random vector in $S^{d-1}$ with the distribution $\lambda/\|\lambda\|$. 
Interestingly, the inverse L\'evy measure, rejection and Bondesson's methods yield this shot noise representation.
While the thinning method may deduce a different shot noise representation, it is significantly elapsed by the representation above in terms of elegance and practicality.

Now, Assumption \ref{sufficient conditions for limiting degeneracy} (c) holds true, due to 
\begin{equation}\label{verification stable law}
 \sigma_m^2=\int_m^{+\infty}\int_{\Omega}(H(r,U(\omega)))^{\otimes 2}\mathbb{P}(d\omega)dr=m^{1-2/\alpha} \frac{\|\lambda\|^{2/\alpha-1}}{2/\alpha-1}\Lambda \to 0,
\end{equation}
as $m\to +\infty$, with $\Lambda:=\int_{S^{d-1}}{\bm \xi}^{\otimes 2}\lambda(d{\bm \xi})$.
If the measure $\lambda$ is not concentrated on a proper linear subspace of $\mathbb{R}^d$, Assumption \ref{sufficient conditions for weak convergence} (a) is satisfied, since then the matrix $\Lambda$ is positive definite.
In addition, Assumption \ref{sufficient conditions for weak convergence} (b) holds true, since for every $\kappa>0$, 
\[
 \int_m^{+\infty}\int_{\Omega}\left[\left\|\sigma_m^{-1}H(r,U(\omega))\right\|^2\right]_{(\kappa,+\infty)}\mathbb{P}(d\omega)dr
 =\int_1^{+\infty}\int_{\Omega}m\left[\frac{(2/\alpha-1)r^{-2/\alpha}\|\lambda\|}{m}\langle {\bm \xi},\Lambda^{-1}{\bm \xi}\rangle \right]_{(\kappa,+\infty)}\frac{\lambda(d{\bm \xi})}{\|\lambda\|}dr\to 0,
\]
as $m\to +\infty$, where the term $[\cdot]_{(\kappa,+\infty)}$ eventually vanishes for almost every $(r,{\bm \lambda})\in (1,+\infty)\times S^{d-1}$.

In the literature, there exist a few variants of the stable law and their error analysis that can proceed in a similar manner.
We refer the reader to, for instance, \cite{houdre2007layered, leguevel2012ferguson} for layered stable and multistable L\'evy processes.

\subsection{Tempered stable law}
Consider again the L\'evy measure \eqref{stable Levy measure}, where $q(\cdot,{\bm \xi})$ here is completely monotone with $q(0+,{\bm \xi})=1$ and $\lim_{r\to +\infty}q(r,{\bm \xi})=0$ for ${\bm \xi}\in S^{d-1}$.
Then, it is the L\'evy measure of a (proper) tempered stable law \cite{rosinski2007tempering}. 
The most well known representation is the one developed in \cite{rosinski2007tempering}, which can be described as 
\[
 H(r,{\bf u})=\left[\left(\frac{r}{\|\lambda\|}\right)^{-1/\alpha}\land \frac{u_1 u_2^{1/\alpha}}{\|u_3\|}\right]\frac{u_3}{\|u_3\|},
\]
where ${\bf u}=(u_1,u_2,u_3)$ corresponds to a random vector taking values in $(0,+\infty)\times [0,1]\times \mathbb{R}^d_0$ with the law $e^{-u_1}du_1 \otimes du_2 \otimes (Q/\|\lambda\|)$, such that $Q(B):=\int_{S^{d-1}}\int_{0+}^{+\infty}\mathbbm{1}_B(r{\bm \xi})Q(dr;{\bm \xi})\lambda(d{\bm \xi})$ and $q(r,{\bm \xi})=\int_{0+}^{+\infty}e^{-rs}Q(ds;{\bm \xi})$.

Now, it is rather straightforward to verify Assumption \ref{sufficient conditions for limiting degeneracy} (c) and Assumption \ref{sufficient conditions for weak convergence} (a) with the aid of \eqref{verification stable law} for the tempered stable law, whereas, in more general terms, it seems very difficult to go through all the conditions, particularly in higher dimensions.
We refer the reader to \cite[Theorems 2.4 and 2.5]{cohen2007gaussian} for sufficient conditions to ensure Assumption \ref{sufficient conditions for limiting degeneracy} (c) and Assumption \ref{sufficient conditions for weak convergence} (a) and (b) for the L\'evy measure in polar coordinates, just like \eqref{stable Levy measure}, by which the tempered stable law is proved to satisfy the conditions for the Gaussian approximation, as long as the matrix $\Lambda$ is positive definite.
In our notation, if the L\'evy measure $\nu-\nu_m$ of the discarded jumps can be decomposed in the polar form $(\nu-\nu_m)(d{\bf z}) = h_m(dr;\bm\xi)\lambda(d\bm\xi)$ for $(r,\bm\xi)\in(0,+\infty)\times S^{d-1}$, then those sufficient conditions read as follows:
There exists a sequence $\{b_m\}_{m\in\mathbb N}$ in $(0,+\infty)$ such that $\liminf_{m\to+\infty}b_m^{-1} \int_{0+}^{+\infty} r^2 h_m(dr;{\bm \xi})> 0$ for almost every ${\bm \xi}\in S^{d-1}$, as well as such that for every $\kappa>0$, $\lim_{m\to+\infty}b_m^{-2}\int_{\|\mathbf z\|>\kappa b_m} \|\mathbf z\|^2\,(\nu-\nu_m)(d\mathbf z) = 0$ (in fact, without requiring the polar form for the latter).

For many other shot noise representations and error analysis of the tempered stable law, 
we refer the reader to \cite{imai2011finite}.

\section{Kernels}\label{section examples of ID processes}

In what follows, we illustrate a variety of kernels.
The first three examples have $\mathcal T=[0,T]$, so the term $\{R_t(m,n):\,t\in [0,T]\}$ is simply irrelevant.
In the next four examples, since the support of their time integrands are only bounded from above by $T$, we decompose the stochastic integral processes with, for instance, $\mathcal T_n = (-n,T)$.
In the last one, since the support is unbounded $\mathcal{T}=\mathbb{R}$, the time truncation needs to be performed from both sides, for instance, $\mathcal{T}_n=(-n,+n)$.
Moreover, in Section \ref{section nfsm}, we illustrate the relevance of further splitting the $Q$ term into multiple independent components (with disjoint time sets in the sense of \eqref{def of Q tilde}), resulting in different degrees of the limiting H\"older continuity (Theorem \ref{theorem Gaussian approximation} (iii)).

\subsection{L\'evy processes}\label{section Levy process}

We first consider the simple case of a L\'evy process $\{X_t:\,t\in [0,T]\}$, which is an infinitely divisible process with the trivial stochastic integral representation $X_t = \int_0^T\mathbbm1_{[0,t]}(s)\,dL_s$ for $t\in[0,T]$.
As the stochastic integral only integrates over, at most, $[0,T]$, the error term $\{R_t(m,n)\}$ is irrelevant here, that is, the only error component of the approximation is $\{Q_t(m):t\in[0,T]\}$ of the discarded jumps.
Clearly, the kernel $\mathbbm1_{[0,t]}(s)$ is uniformly bounded and square-integrable on $[0,T]$ (Assumptions \ref{sufficient conditions for limiting degeneracy} (a) and (b)).
Moreover, Theorem \ref{theorem Gaussian approximation} (ii) holds as Assumption \ref{sufficient conditions for weak convergence} (c) is satisfied with $c_1=1$.
In fact, the Gaussian approximation for the discarded jumps of L\'evy processes is well known in the literature \cite{cohen2007gaussian}.
We remark that in the case of L\'evy processes, the weak convergence in $\mathcal{D}([0,T];\mathbb{R}^d)$ is equivalent to the weak convergence of a marginal law \cite[Exercise 16.10]{kallenberg2002foundations}.

\subsection{L\'evy-driven Ornstein--Uhlenbeck processes}\label{section Levy-driven OU process}

The L\'evy-driven Ornstein--Uhlenbeck process $\{X_t:t\in[0,T]\}$ is described by the stochastic differential equation $dX_t = \lambda(\mu-X_t)\,dt + dL_{t}$, with $\lambda>0$ and $\mu\in\mathbb R$, where its explicit solution is available as follows:
\begin{equation}\label{OU explicit solution}
	X_t = e^{-\lambda t}X_0 + \mu\left(1-e^{-\lambda t}\right) + \int_0^t e^{-\lambda(t-s)}\,dL_{s}.
\end{equation}
In light of the representation \eqref{OU explicit solution}, it is an infinitely divisible process with the kernel $f(t,s)=e^{-\lambda(t-s)}\mathbbm1_{[0,t)}(s)$, that is,
it only integrates over a compact time interval.
Hence, we have $\mathcal{T}_n\equiv \mathcal{T}$ and the sole error component is $\{Q_t(m):t\in[0,T]\}$, corresponding to the discarded jumps. 
The kernel is clearly square-integrable (Assumption \ref{sufficient conditions for limiting degeneracy} (b)) and is uniformly bounded (Assumption \ref{sufficient conditions for limiting degeneracy} (a)).
Moreover, it satisfies Assumption \ref{sufficient conditions for weak convergence} (c) with $c_1=1$ for Theorem \ref{theorem Gaussian approximation} (ii), as
\[
 \int_0^T (f(t_2,s)-f(t_1,s))^2\,ds = \frac{1}{2\lambda}\left[2\left(1-e^{-\lambda(t_2-t_1)}\right) + e^{-2\lambda t_1}\left(1 - e^{-\lambda(t_2 - t_1)}\right)^2\right] \le \frac32(t_2 - t_1),\quad 0\le t_1\le t_2\le T,
\]
where the inequality holds as $(1 - e^{-x})^2 \le 1 - e^{-x} \le x$ for all $x\ge0$.

\subsection{A fractional L\'evy motion}\label{section ftsm}

Consider the following kernel defined on a bounded time interval:
\[
	K_{H,\alpha}(t,s)\coloneqq c_{H,\alpha}\left[\left(\frac{t}{s}\right)^{H-1/\alpha}(t-s)^{H-1/\alpha} - \left(H-\frac{1}{\alpha}\right)s^{1/\alpha - H}\int_s^t u^{H - 1/\alpha - 1}(u-s)^{H-1/\alpha}\,du\right]\mathbbm1_{[0,t)}(s),
\]
where $\alpha\in(0,2)$, $H\in (1/\alpha-1/2,1/\alpha+1/2)$ and $c_{H,\alpha}$ is a suitable constant. 
We highlight that when $H\in(1/\alpha-1/2,1/\alpha)$, it holds that $\lim_{s\to t-}|K_{H,\alpha}(t,s)| = +\infty$, thus leading to failure of Assumption \ref{sufficient conditions for limiting degeneracy} (a) and thus sample unboundedness.
Hence, it suffices to focus on $H\in (1/\alpha,1/\alpha + 1/2)$.
Then, there exists a continuous modification of the infinitely divisible process which is almost surely locally H\"older continuous with exponent $\gamma < H-1/\alpha$.
Interesting features of this kernel are its self-similarity $K_{H,\alpha}(ht,s) = h^{H-1/\alpha}K_{H,\alpha}(t,s/h)$ for all $h>0$, as well as the second order one \cite[Lemmas 2.1 and 2.3]{houdre2006fractional}:
	\[
		\int_0^T\left(K_{H,\alpha}(t_2,s) - K_{H,\alpha}(t_1,s)\right)^2\,ds = c_{H,\alpha}(t_2-t_1)^{2H -2/\alpha+1},\quad 0\le t_1\le t_2\le T,
	\]
which verifies Assumption \ref{sufficient conditions for limiting degeneracy} (b) and Assumption \ref{sufficient conditions for weak convergence} (c) with $c_1=2H -2/\alpha+1$.
Thus, by Theorem \ref{theorem Gaussian approximation} (iii), the limiting Gaussian process of the scaled term $\{\widetilde{Q}_t(m):\,t\in [0,T]\}$ is almost surely locally H\"older continuous with exponent in $(0,H-1/\alpha)$.
We refer the reader to \cite{houdre2006fractional} for typical sample paths with various parameter sets generated by shot noise representations.

\subsection{Linear fractional L\'evy motions}\label{section nfsm}

Let $n\in \mathbb{N}$ and $\alpha\in(0,2)$.
The $n$-th order moving average kernel with Hurst parameter $H\in(n-1,n)\backslash\{1/\alpha\}$, such that $H-1/\alpha$ is not an integer, is given by
\[
f_n(t,s;H,\alpha)\coloneqq\frac{1}{\Gamma(H-1/\alpha + 1)}\left((t-s)_+^{H-1/\alpha} - \sum_{k=0}^{n-1}\binom{H-1/\alpha}{k}t^k(-s)_+^{H-1/\alpha-k}\right),
\quad s\in \mathbb{R}.
\]
If $H=1/\alpha$, then the kernel reduces to $\mathbbm{1}_{[0,t)}(s)$ by the zero-power convention, corresponding to L\'evy processes (Section \ref{section Levy process}).
If $n=1$, then it is the kernel for typical linear fractional L\'evy motions, whereas it is a higher order one \cite{kawai2016higher} with $n\in \{2,\cdots\}$.
Much like the kernel of Section \ref{section ftsm}, this moving average kernel captures self-similarity, with strong self-similarity with index $H$ if the integrator is self-similar with index $\alpha$, and second-order self-similarity in the case where the integrator is not necessarily self-similar but has finite second-order moments (see, for instance, \cite{carnaffan2019analytic}).

It is easy to show that Assumption \ref{sufficient conditions for limiting degeneracy} (a) is satisfied only when $H-1/\alpha > n-1$, while Assumption \ref{sufficient conditions for limiting degeneracy} (b) is satisfied only when $H-1/\alpha\in (n-3/2,n-1/2)$.
Hence, it suffices to focus on $H-1/\alpha \in (n-1,n-1/2)$ for our simulation purpose.
If $n=1$, on the one hand, then it holds that for $0\le t_1\le t_2\le T$,
\begin{equation}\label{lips 5.4}
 \int_{\mathbb{R}}\left(f_1(t_2,s;H,\alpha)-f_1(t_1,s;H,\alpha)\right)^2ds
 =\frac{(t_2-t_1)^{2H-2/\alpha+1}}{(\Gamma(H-1/\alpha + 1))^2}\int_{\mathbb{R}}\left((1-s)_+^{H-1/\alpha}-(-s)_+^{H-1/\alpha}\right)^2ds,
\end{equation}
which is finite valued since $H-1/\alpha \in (0,+1/2)$ when $n=1$.

On the other hand, if $n\in \{2,\cdots\}$, then the kernel is no longer as uniform as \eqref{lips 5.4} over the entire domain $\mathcal{T}$.
We split $\mathcal{T}=\mathbb{R}$ into the two disjoint and exhaustive sets $\mathcal{S}_1=(-\infty,0]$ and $\mathcal{S}_2=(0,+\infty)$ so that 
\[
\widetilde{Q}_t(m;\mathcal{S}_k) =\left(\int_{-\infty}^0+\int_0^{+\infty}\right) \int_{\mathbb R_0}f_n(t,s;H,\alpha)\sigma_m^{-1}\mathbf z\,((\mu - \mu_m) - (\nu - \nu_m)\mathbbm1_{(0,1]}(\|\mathbf z\|))(d\mathbf z,ds),
\]
where the two integrals $\int_{-\infty}^0$ and $\int_0^{+\infty}$ correspond to $\mathcal{S}_1$ and $\mathcal{S}_2$, respectively.
Clearly, the resulting two stochastic processes are independent of each other and thus can be treated separately.
First, on the negative time line $(-\infty,0](=\mathcal{S}_1\subset \mathcal{T})$, we employ the mean value theorem with the aid of the recurrence formula $\partial_t f_n(t,s;H,\alpha)=f_{n-1}(t,s;H-1,\alpha)$ for $t>s$, so as to yield that for $0\le t_1\le t_2\le T$,
\[
  \int_{-\infty}^0 \left(f_n(t_2,s;H,\alpha)-f_n(t_1,s;H,\alpha)\right)^2ds=
  (t_2-t_1)^2 \int_{-\infty}^0\left(f_{n-1}(\Theta_s(t_1,t_2),s;H-1,\alpha)\right)^2ds\le (t_2-t_1)^2 \int_{\mathbb{R}}\left(f_{n-1}(t_2,s;H-1,\alpha)\right)^2ds,
\]
where $\Theta_s(t_1,t_2)$ is a real number in the interval $(t_1,t_2)$ depending on $s$.
The inequality holds since $f_{n-1}(\cdot,s;H-1,\alpha)$ is positive and increasing on $[0,T]$ as well as the integral is finite valued, both due to $H-1/\alpha \in (n-1,n-1/2)$ and $n\in \{2,\cdots\}$.
Therefore, the first component (corresponding to $\mathcal{S}_1\subset \mathcal{T}$) satisfies Assumption \ref{sufficient conditions for weak convergence} (c) with $c_1=2$, that is, Theorem \ref{theorem Gaussian approximation} (ii) with H\"older exponent in $(0,1/2)$.
In turn, since the terms $(-s)_+^{H-1/\alpha-k}$ all vanish on the positive time line $(0,+\infty)(=\mathcal{S}_2\subset \mathcal{T})$, it holds that for $0\le t_1\le t_2\le T$,
\begin{align*}
  \int_0^{+\infty} \left(f_n(t_2,s;H,\alpha)-f_n(t_1,s;H,\alpha)\right)^2ds
 &\le \int_{\mathbb{R}} \frac{((t_2-s)_+^{H-1/\alpha}-(t_1-s)_+^{H-1/\alpha})^2}{(\Gamma(H-1/\alpha + 1))^2}ds\\
 &=\frac{(t_2-t_1)^{2H-2/\alpha+1}}{(\Gamma(H-1/\alpha + 1))^2}\int_{\mathbb{R}}\left((1-s)_+^{H-1/\alpha}-(-s)_+^{H-1/\alpha}\right)^2ds,
\end{align*}
again due to $H-1/\alpha \in (n-1,n-1/2)$ and $n\in \{2,\cdots\}$.
Therefore, the second component (corresponding to $\mathcal{S}_2\subset \mathcal{T}$) satisfies Assumption \ref{sufficient conditions for weak convergence} (c) with $c_1=2H-2/\alpha+1$, that is, Theorem \ref{theorem Gaussian approximation} (ii) with H\"older exponent in $(0,H-1/\alpha)$.

For typical sample paths based on the proposed simulation method and error analysis, we refer the reader to \cite{kawai2016higher} and \cite{carnaffan2019analytic} when the integrator is, respectively, stable and tempered stable.

\subsection{L\'evy-driven CARMA processes}\label{section CARMA process}

We now turn to 
L\'evy-driven CARMA processes \cite{kawai2017sample}.
To define this class, fix $a_1,\cdots,a_p,b_0,\cdots,b_{p-1}\in\mathbb R$ such that $b_q = 1$, $q\le p-1$ and $b_k = 0$ for $k>q$, and define the polynomials $a(z) \coloneqq z^p + a_1z^{p-1} + \cdots + a_p$ and $b(z)\coloneqq b_0 + b_1 z + \cdots + b_qz^q$ in such a way that $a(z)$ and $b(z)$ have no common roots. Define $A\in\mathbb R^{p\times p}$ by
\[
A \coloneqq \begin{bmatrix}
0 & 1 & 0 & \cdots & 0\\
0 & 0 & 1 & \cdots & 0\\
\vdots & \vdots & \vdots & \cdots & \vdots\\
0 & 0 & 0 & \cdots & 1\\
-a_p & -a_{p-1} & -a_{p-2} & \cdots & -a_1
\end{bmatrix},
\]
and denote the eigenvalues of $A$ as $\lambda_1,\cdots,\lambda_p\in\mathbb R$, that is, $a(z) = \prod_{k=1}^p(z-\lambda_k)$.
We assume that the roots of $a(z)$ are all real, strictly negative and distinct.
Define $e_p\in\mathbb R^p$ as the unit vector in the $p$-th direction and $b\coloneqq[b_1, b_1,\cdots,b_{p-1}]^\intercal\in\mathbb R^p$.
Then, a L\'evy-driven CARMA process in $\mathbb R$ of order $(p,q)$ with $p>q$ is defined as $\{Y_t:t\in\mathbb R\}$, where $Y_t\coloneqq\langle b,X_t\rangle$ and $\{X_t:t\in\mathbb R\}$ is a stochastic process in $\mathbb R^p$ satisfying
\[
X_{t_2} = e^{A(t_2-t_1)}X_{t_1} + \int_{t_1}^{t_2} e^{A(t_2 - s)}e_p\,dL_s,\quad t_1\le t_2.
\]
Under suitable technical conditions, the L\'evy-driven CARMA process $\{Y_t:t\in\mathbb R\}$ can be expressed as a linear combination of dependent Ornstein-Uhlenbeck-like processes as follows:
\[
Y_t = \sum_{k=1}^p\frac{b(\lambda_k)}{a'(\lambda_k)}\int_{-\infty}^t e^{\lambda_k(t-s)}\,dL_s,\quad t\in [0,T].
\]
Hence, the kernel here is given by
\[
 f(t,s)=\sum_{k=1}^p\frac{b(\lambda_k)}{a'(\lambda_k)} e^{\lambda_k(t-s)}\mathbbm{1}_{(-\infty,t)}(s),\quad s\in\mathbb{R},
\]
that is, a linear combination of bounded exponential functions.
Clearly, the kernel satisfies boundedness and square-integrability (Assumption \ref{sufficient conditions for limiting degeneracy} (a) and (b)), as well as the Lipschitz continuity (Assumption \ref{sufficient conditions for weak convergence} (c)), since
\[
 \int_{-\infty}^T \left(e^{\lambda (t_2-s)}-e^{\lambda (t_1-s)}\right)^2\,ds = \frac{1-e^{\lambda (t_2-t_1)}}{-\lambda} \le t_2 - t_1,\quad 0\le t_1\le t_2\le T,
\]
for $\lambda<0$.
Hence, the kernel satisfies all the relevant conditions for the limiting degeneracy (Proposition \ref{sufficient conditions for limiting degeneracy}) and the Gaussian approximation (Theorem \ref{theorem Gaussian approximation} (ii) with $\mathcal{S}=\mathcal{T}$).
We refer the reader to \cite{kawai2017sample} for the presentation of typical sample paths based on the proposed simulation method and error analysis when the integrator is stable and gamma.

\subsection{L\'evy-driven reverse Ornstein-Uhlenbeck processes}\label{section reverse OU}

The L\'evy-driven reverse Ornstein-Uhlenbeck process is defined as the stochastic integral \eqref{Levy-driven stochastic integral} with the kernel $f(t,s) = e^{-\lambda(s-t)}\mathbbm1_{[t,+\infty)}(s)$, with $\lambda>0$.
Therefore, the kernel satisfies boundedness and square-integrability (Assumption \ref{sufficient conditions for limiting degeneracy} (a) and (b)), as well as the Lipschitz continuity (Assumption \ref{sufficient conditions for weak convergence} (c)):
\[
 \int_{\mathbb R}(f(t_2,s)-f(t_1,s))^2\,ds 
	= \frac1\lambda \left(1-e^{-\lambda(t_2-t_1)}\right)\le t_2-t_1,\quad 0\le t_1\le t_2\le T,
\]
satisfying the conditions for Proposition \ref{sufficient conditions for limiting degeneracy} and Theorem \ref{theorem Gaussian approximation} (ii) with $\mathcal{S}=\mathcal{T}$.

\subsection{Log-fractional L\'evy motion}\label{section log fractional}

The log-fractional L\'evy motion is defined as the stochastic integral \eqref{Levy-driven stochastic integral} with the kernel $f(t,s) = \ln |t-s| - \ln|s|$ for $s\in \mathbb{R}$.
Despite that the kernel is square-integrable (Assumption \ref{sufficient conditions for limiting degeneracy} (b)), the log-fractional motion is known to be unbounded on every interval of positive length \cite[Example 10.2.6]{samorodnitsky1994stable}.
Indeed, Assumption \ref{sufficient conditions for limiting degeneracy} (a) fails.
 

\subsection{Real harmonizable fractional motions}\label{section real harmonizable}

The real harmonizable fractional motion \cite{benassi2002, samorodnitsky1994stable} is a real-valued random field with locally H\"older continuous sample paths and can generalize fractional Gaussian fields.
It is within our scope if the time index is one dimensional with $\mathcal{T}=\mathbb{R}$ and, for instance, $\mathcal{T}_n=(-n,+n)$.
On a compact time interval $[0,T]$, the real harmonizable fractional L\'evy motion is formulated as
\[
 X_t=\int_{\mathbb{R}}\frac{e^{-its}-1}{|s|^{1/\alpha+H}}M(ds),\quad t\in [0,T],
\]
where $H\in (0,1)$, $\alpha\in (0,2]$ and $M$ is a suitable random L\'evy measure on $\mathbb{R}$.
On the one hand, the kernel is square-integrable (Assumption \ref{sufficient conditions for limiting degeneracy} (b)) and satisfies the regularity condition (Assumption \ref{sufficient conditions for weak convergence} (c) with $\mathcal{S}=\mathcal{T}$ and $c_1=2H+2/\alpha-1$), due to
\begin{align*}
 \int_{\mathbb{R}}\left|\frac{e^{-it_1s}-1}{|s|^{1/\alpha+H}}-\frac{e^{-it_2s}-1}{|s|^{1/\alpha+H}}\right|^2ds
 = (t_2-t_1)^{2H+2/\alpha-1}\int_{\mathbb{R}}\frac{2(1-\cos(s))}{|s|^{2/\alpha+2H}}ds,\quad 0\le t_1\le t_2\le T,
\end{align*}
if and only if $H+1/\alpha\in (1/2,3/2)$.
On the other hand, Assumption \ref{sufficient conditions for limiting degeneracy} (a) fails when $H+1/\alpha >1$, since then $\sup_{s\in \mathbb{R}}|f(t,s)|=+\infty$ for all $t\in [0,T$].
Hence, for our purposes, it suffices to focus on $H+1/\alpha \in (1,3/2)$. 
Since $c_1=2H+2/\alpha-1\in (1,2)$, Theorem \ref{theorem Gaussian approximation} (ii) holds, provided that the random L\'evy measure satisfies all the relevant conditions in Assumptions \ref{sufficient conditions for limiting degeneracy} and \ref{theorem Gaussian approximation}. 
We refer the reader to \cite{dejean2005fracsim, lacaux2004series, samorodnitsky1994stable} for other simulation methods, error analysis and typical sample paths.

\subsection{Pareto tails of an error term with the stable integrator}\label{subsection R}

We close this study by illustrating a possible issue when the integrator is a stable process.
We have already seen in Example \ref{example stable law} that the stable L\'evy measure \eqref{stable Levy measure} (with $q\equiv 1$) satisfies the conditions for the limiting degeneracy (Assumption \ref{sufficient conditions for limiting degeneracy} (c)), and moreover for the Gaussian approximation (Assumption \ref{sufficient conditions for weak convergence} (a) and (b)), that is, the $Q$ term is asymptotically as quiet as a Gaussian process, while the $R$ term, if exists at all, vanishes in the limit.
A possible issue here is, however, that the $R$ term exhibits non-negligible Pareto tails (for instance, \cite[Theorem 10.5.1]{samorodnitsky1994stable}) in the sense of 
\[
 \lim_{\theta\to+\infty}\theta^\alpha\mathbb P\left(\sup_{t\in[0,T]}\left|\lim_{m\to +\infty}R_t(m,n)\right| > \theta\right) = c\int_{\mathcal{T}\setminus \mathcal{T}_n}\sup_{t\in[0,T]}|f(t,s)|^\alpha\,ds,\quad n\in \mathbb{N},
\]
where $c$ is a suitable positive constant, independent of $n$.
This is the case whenever the integral of the righthand side is finite valued, including the stable CARMA process \cite[Theorem 4.1]{kawai2017sample} and the higher order fractional stable motion \cite[Theorem 6.4]{kawai2016higher}.
It is then desirable to take the time truncation $\mathcal{T}_n$ sufficiently large to suppress the tails, taking into account a typical tradeoff with additional computing cost for more summands in light of \eqref{principal truncation}.

\small
	\bibliographystyle{abbrv}
	\bibliography{master}

\appendix

\section{Proofs}\label{appendix proofs}

Define $\varphi_a({\bf x};{\bf y}):=e^{i\langle {\bf y},{\bf x}\rangle}-1-i\langle {\bf y},{\bf x}\rangle \mathbbm{1}_{(0,1]}(\|{\bf x}\|)$ and $\varphi_b({\bf x};{\bf y}):=e^{i\langle {\bf y},{\bf x}\rangle}-1-i\langle {\bf y},{\bf x}\rangle$.
Note that
\begin{equation}\label{varphi a b}
	\varphi_a({\bf x};{\bf y})=\varphi_b({\bf x};{\bf y})+i\langle {\bf y},{\bf x}\rangle \mathbbm{1}_{(1,+\infty)}(\|{\bf x}\|),\quad 
	|\varphi_a({\bf x};{\bf y})|\le \frac{1}{2}\|{\bf y}\|^2\|{\bf x}\|^2\mathbbm{1}_{(0,1]}(\|{\bf x}\|)+2\times \mathbbm{1}_{(1,+\infty)}(\|{\bf x}\|),\quad 
	 \left|\varphi_b({\bf x};{\bf y})\right|\le \frac{1}{2}|\langle {\bf x},{\bf y}\rangle|^2,
\end{equation}
for all $({\bf x},{\bf y})\in \mathbb{R}^d\times \mathbb{R}^d$.
Recall the notation $[c]_B:=c\mathbbm{1}_B(\|c\|)$ for $c\in\mathbb{R}^q$ and $B\in\mathcal{B}(0,+\infty)$ for general dimension $q\in\mathbb{N}$.

\begin{proof}[Proof of Proposition \ref{proposition limiting degeneracy}]
	Fix $l\in \mathbb{N}$, let $\{t_k\}_{k\in\{1,\cdots,l\}}$ be a sequence of constants in $[0,T]$, and let $\{\theta_k\}_{k\in\{1,\cdots,l\}}$ be a sequence of constants in $\mathbb{R}$.
	To ease the notation, we denote $\psi(s):=\sum_{k=1}^l \theta_k  f(t_k,s)$.
	For the convergence of $\{Q_t(m):\,t\in [0,T]\}$ in $m$, it suffices to show the pointwise convergence: for every ${\bf y}\in \mathbb{R}^d$,
	\[
	\mathbb{E}\left[\exp\left[i\left\langle {\bf y},\sum\nolimits_{k=1}^l \theta_k Q_{t_k}(m)\right\rangle \right]\right]=\exp\left[\int_{\mathcal{T}}\int_{\mathbb{R}_0^d}\varphi_a\left(\psi(s){\bf z};{\bf y}\right)(\nu-\nu_m)(d{\bf z})ds\right]\to 1,
\]
	as $m\to +\infty$.
	Fix ${\bf y}\in\mathbb{R}^d$.
	It holds by \eqref{varphi a b} that  
	\begin{align*}
		\left|\int_{\mathcal{T}}\int_{\mathbb{R}_0^d}\varphi_a\left(\psi(s){\bf z};{\bf y}\right)(\nu-\nu_m)(d{\bf z})ds\right|
		&\le \int_{\mathcal{T}}\int_m^{+\infty}\int_{\Omega}\left|\varphi_b\left(\psi(s)H(r,U(\omega));{\bf y}\right)\right|\mathbb{P}(d\omega)dr\,ds\\
		&\qquad + \|{\bf y}\|\int_{\mathcal{T}}\int_1^{+\infty}\int_{\Omega}m\left[\left\|\psi(s)H(mr,U(\omega))\right\|\right]_{(1,+\infty)}\mathbb{P}(d\omega)dr\,ds,
	\end{align*}
	where the first term tends to zero, since 
	\begin{align*}
		\int_{\mathcal{T}}\int_m^{+\infty}\int_{\Omega}\left|\varphi_b\left(\psi(s)H(r,U(\omega));{\bf y}\right)\right|\mathbb{P}(d\omega)dr\,ds
		& \le \frac{1}{2}\int_{\mathcal{T}}\int_m^{+\infty}\int_{\Omega}\left|\left\langle {\bf y}, \psi(s)H(r,U(\omega))\right\rangle \right|^2\mathbb{P}(d\omega)dr\,ds \\
		&=\frac{1}{2}\left\langle {\bf y},\int_{\mathcal{T}}\psi(s)\sigma_m^2(\psi(s))^{\top}ds \,{\bf y}\right\rangle\to 0,
	\end{align*}
	where the passage to the limit can be justified by the dominated convergence theorem due to Assumption \ref{sufficient conditions for limiting degeneracy} (a) and (b).
	For the second term, the integrand $m\left[\left\|\psi(s)H(mr,U(\omega))\right\|\right]_{(1,+\infty)}$ eventually vanishes for $(\mathbb{P}\otimes {\rm Leb}\otimes {\rm Leb})$-$a.e.$ $(\omega,r,s)\in \Omega\times (1,+\infty)\times \mathcal{T}$, due to Assumption \ref{sufficient conditions for limiting degeneracy} (c).
	

Next, noting that the limiting process $\{\lim_{m\to +\infty}R_t(m,n):\,t\in [0,T]\}(=:\{R_t(n):\,t\in [0,T]\})$ exists, it suffices, in a similar manner to (i), to observe that 
\[
\mathbb{E}\left[\exp\left[i\langle {\bf y},\sum\nolimits_{k=1}^l \theta_k R_{t_k}(n)\rangle \right]\right]=\exp\left[\int_{\mathcal{T}\setminus \mathcal{T}_n}\int_{\mathbb{R}_0^d}\varphi_a\left(\psi(s){\bf z};{\bf y}\right)\nu(d{\bf z})ds\right]\to 1,
\]
for ${\bf y}\in \mathbb{R}^d$, due to $\cup_n \mathcal{T}_n= \mathcal{T}$.
\end{proof}

\begin{proof}[Proof of Theorem \ref{theorem Gaussian approximation}]
	Throughout, we let $m$ be large enough to satisfy Assumption \ref{sufficient conditions for weak convergence} (a).
	
For (i), fix $l\in \mathbb{N}$, let $\{t_k\}_{k\in\{1,\cdots,l\}}$ be a sequence of constants in $[0,T]$, and let $\{\theta_k\}_{k\in\{1,\cdots,l\}}$ be a sequence of constants in $\mathbb{R}$.
	It suffices to show the pointwise convergence: for every ${\bf y}\in \mathbb{R}^d$,
	\begin{align*}
		\ln \mathbb{E}\left[\exp\left[i\left\langle {\bf y},\sum\nolimits_{k=1}^l \theta_k \widetilde{Q}_{t_k}(m;\mathcal{T})\right\rangle \right]\right]
		&=\int_{\mathcal{T}}\int_{\mathbb{R}_0^d}\varphi_b\left(\psi(s)\sigma_m^{-1}{\bf z};{\bf y}\right)(\nu-\nu_m)(d{\bf z})ds\\
		&\qquad +i\left\langle {\bf y},\int_{\mathcal{T}}\int_{\mathbb{R}_0^d}\left[\psi(s)\sigma_m^{-1}{\bf z}\right]_{(1,+\infty)} (\nu-\nu_m)(d{\bf z})ds\right\rangle\\
		&\to -\frac{1}{2}\left\langle {\bf y},\int_{\mathcal{T}}(\psi(s))^{\otimes 2}ds \,{\bf y}\right\rangle,
	\end{align*}
	where $\psi(s):=\sum_{k=1}^l \theta_k  f(t_k,s)$, as in the proof of Theorem \ref{proposition limiting degeneracy}.
	Hereafter, we fix ${\bf y}\in\mathbb{R}^d$ throughout.
	First, we show that
	\begin{align*}
		\int_{\mathcal{T}}\int_{\mathbb{R}_0^d}\varphi_b\left(\psi(s)\sigma_m^{-1}{\bf z};{\bf y}\right)(\nu-\nu_m)(d{\bf z})ds
		&=\int_{\mathcal{T}}\int_m^{+\infty}\int_{\Omega}\varphi_b\left(\psi(s)\sigma_m^{-1}H(r,U(\omega));{\bf y}\right)\mathbb{P}(d\omega)dr\,ds\\
		&\sim -\frac{1}{2}\int_{\mathcal{T}}\int_1^{+\infty}\int_{\Omega}m\left|\left\langle {\bf y},\psi(s)\sigma_m^{-1}H(mr,U(\omega))\right\rangle\right|^2\mathbb{P}(d\omega)dr\,ds\\
		&=-\frac{1}{2}\left\langle {\bf y},\int_{\mathcal{T}}(\psi(s))^{\otimes 2}ds \,{\bf y}\right\rangle,
	\end{align*}
	where the asymptotic equivalence remains to be justified.
	Assumption \ref{sufficient conditions for weak convergence} (a) reads that for every $\kappa >0$,
	\[
	\int_{\|\sigma_m^{-1}{\bf z}\|^2 >\kappa} \left\|\sigma_m^{-1}{\bf z}\right\|^2 (\nu-\nu_m)(d{\bf z})
	=\int_1^{+\infty}\int_{\Omega} \left[\left\|\sigma_m^{-1}H(mr,U(\omega))\right\|^2\right]_{(\kappa,+\infty)} \mathbb{P}(d\omega)dr \to 0, 
	\]
	as $m\to +\infty$.
	This, along with Assumption \ref{sufficient conditions for limiting degeneracy} (a), ensures that for each $\epsilon>0$, there exists $m_{\epsilon}\in \mathbb{N}$ such that for every $m\ge m_{\epsilon}$,
	\[
	\left|\frac{\varphi_b(\psi(s)\sigma_m^{-1}H(mr,U(\omega));{\bf y})}{-\frac{1}{2}|\langle {\bf y},\psi(s)\sigma_m^{-1}H(mr,U(\omega))\rangle|^2}-1\right|<\epsilon,\quad (\mathbb{P}\otimes {\rm Leb}\otimes {\rm Leb})\text{-}a.e.\, (\omega,r,s)\in \Omega\times (1,+\infty)\times \mathcal{T},
	\]
	that is, for every $m\ge m_{\epsilon}$,
	\begin{multline*}
		\int_{\mathcal{T}}\int_1^{+\infty}\int_{\Omega}\left|m\varphi_b\left(\psi(s)\sigma_m^{-1}H(mr,U(\omega));{\bf y}\right)+\frac{1}{2}m\left|\left\langle {\bf y},\psi(s)\sigma_m^{-1}H(mr,U(\omega))\right\rangle\right|^2\right|\mathbb{P}(d\omega)dr\,ds\\
		\le \frac{1}{2}\epsilon \int_{\mathcal{T}}\int_1^{+\infty}\int_{\Omega}m\left|\left\langle {\bf y},\psi(s)\sigma_m^{-1}H(mr,U(\omega))\right\rangle\right|^2\mathbb{P}(d\omega)dr\,ds
		= \frac{1}{2}\epsilon\left\langle {\bf y},\int_{\mathcal{T}}(\psi(s))^{\otimes 2}ds \,{\bf y}\right\rangle,
	\end{multline*}
	again due to \eqref{varphi a b}, which justifies the desired asymptotic equivalence since $\epsilon$ can be chosen arbitrarily small.
	It remains to observe that  
	\begin{align*}
		\left\|\int_{\mathcal{T}}\int_{\mathbb{R}_0^d}\left[\psi(s)\sigma_m^{-1}{\bf z}\right]_{(1,+\infty)} (\nu-\nu_m)(d{\bf z})ds\right\|
		&\le \int_{\mathcal{T}}\int_m^{+\infty}\int_{\Omega}\left[\left\|\psi(s)\sigma_m^{-1}H(r,U(\omega))\right\|\right]_{(1,+\infty)} \mathbb{P}(d\omega)dr\,ds\\
		&\le \int_{\mathcal{T}}\int_1^{+\infty}\int_{\Omega}\left[\left\|\psi(s)\sigma_m^{-1}H(mr,U(\omega))\right\|^2\right]_{(1,+\infty)} \mathbb{P}(d\omega)dr\,ds,
	\end{align*}
which tends to zero as $m\to +\infty$, due to Assumption \ref{sufficient conditions for limiting degeneracy} (a) and Assumption \ref{sufficient conditions for weak convergence} (b).
	
For (ii) and (iii), recall that the time domain here is $\mathcal{S}$, which is a subset of $\mathcal{T}$ of (i).
First, consider a c\`adl\`ag version of $\{\widetilde{Q}_t(m;\mathcal{S}):\,t\in [0,T]\}$.
	Let $\{\tau_m\}_{m\in\mathbb{N}}$ be a sequence of stopping times (with respect to the filtration generated by $\{\widetilde{Q}_t(m;\mathcal{S}):t\in [0,T]\}$) taking values in $[0,T]$ and let $\{h_m\}_{m\in\mathbb{N}}$ be a sequence of positive constants decreasing to zero.
	Then, the desired result follows from \cite[Theorem 16.11]{kallenberg2002foundations}, if 
	\[
	\left\|\widetilde{Q}_{\tau_m+h_m}(m;\mathcal{S})-\widetilde{Q}_{\tau_m}(m;\mathcal{S})\right\|=\left\|\int_{\mathcal{S}}\int_{\mathbb{R}_0^d}\left(f(\tau_m+h_m,s)-f(\tau_m,s)\right)\sigma_m^{-1}{\bf z}((\mu-\mu_m)-(\nu-\nu_m)\mathbbm{1}_{(0,1]}(\|{\bf z}\|))(d{\bf z},ds)\right\|\stackrel{\mathbb{P}}{\to} 0,
	\]
as $m\to +\infty$.
	To show that the last convergence holds true, it suffices to observe that 
	\begin{align*}
	&\mathbb{E}\left[\left\|\int_{\mathcal{S}}\int_{\mathbb{R}_0^d}\left(f(\tau_m+h_m,s)-f(\tau_m,s)\right)\sigma_m^{-1}{\bf z}((\mu-\mu_m)-(\nu-\nu_m)(d{\bf z},ds)\right\|^2\right]\\
	&\qquad =\mathbb{E}\left[\int_{\mathcal{S}}\int_{\mathbb{R}_0^d}\left\|\left(f(\tau_m+h_m,s)-f(\tau_m,s)\right)\sigma_m^{-1}{\bf z}\right\|^2(\nu-\nu_m)(d{\bf z})ds\right]\\
	&\qquad =\mathbb{E}\left[\int_{\mathbb{R}_0^d}(\sigma_m^{-1}{\bf z})^{\top}\left(\int_{\mathcal{S}}\left(f(\tau_m+h_m,s)-f(\tau_m,s)\right)^{\otimes 2}ds\right)(\sigma_m^{-1}{\bf z})(\nu-\nu_m)(d{\bf z})\right]\le c_2 h_m^{c_1}d\to 0,
	\end{align*}
	where we have applied the Wiener-Ito isometry, Assumption \ref{sufficient conditions for weak convergence} (c) and the identity
	\[
	\int_{\mathbb{R}_0^d}\left\langle {\bf z}, (\sigma^2_m)^{-1}{\bf z}\right\rangle(\nu-\nu_m)(d{\bf z})=\int_{\mathbb{R}_0^d}\|{\bf z}\|^2\rho_m(d{\bf z})={\rm tr}\left(\int_{\mathbb{R}_0^d}{\bf z}^{\otimes 2}\rho_m(d{\bf z})\right)={\rm tr}(\mathbb{I}_d)=d,
	\]
	where $\rho_m=(\nu-\nu_m)\circ \sigma_m$ denotes the push forward of the L\'evy measure $(\nu-\nu_m)$ by the map ${\bf z}\to \sigma_m^{-1}{\bf z}$.

	Finally, for a continuous version of $\{\widetilde{Q}_t(m;\mathcal{S}):\,t\in [0,T]\}$, it holds, as before, that for $0\le t_1\le t_2\le T$,
	\[
		\mathbb{E}\left[\left\|\int_{\mathcal{S}}\int_{\mathbb{R}_0^d}\left(f(t_2,s)-f(t_1,s)\right)\sigma_m^{-1}{\bf z}((\mu-\mu_m)-(\nu-\nu_m)(d{\bf z},ds)\right\|^2\right]
		\le c_2(t_2-t_1)^{c_1}d,
	\]
	where the last term is independent of $m$.
	Here, 	It remains to show that the deterministic residual term tends to zero uniformly, as follows:
	\[
	\sup_{t\in [0,T]}\left\|\int_{\mathcal{S}}\int_{\mathbb{R}_0^d}f(t,s)\sigma_m^{-1}{\bf z}(\nu-\nu_m)\mathbbm{1}_{(1,+\infty)}(\|{\bf z}\|))(d{\bf z},ds)\right\|
	\le \int_{\mathcal{S}}\int_1^{+\infty}\int_{\Omega} \sup_{t\in [0,T]}\left\|f(t,s)m\sigma_m^{-1}[H(mr,U(\omega))]_{(1,+\infty)}\right\|\mathbb{P}(d\omega)dr\,ds,
	\]
which tends to zero as $m\to +\infty$, since the term $[H(mr,U(\omega))]_{(1,+\infty)}$ eventually vanishes for almost every $(\omega,r)$, as well as due to Assumption \ref{sufficient conditions for limiting degeneracy} (a).
	Hence, the claim holds by \cite[Corollary 16.9]{kallenberg2002foundations}.	
\end{proof}

\end{document}